\documentclass[11pt,a4paper]{article}
\usepackage{amsmath, amsfonts, amssymb, amsthm,epsfig}
 \usepackage[ps2pdf,linktocpage,bookmarks,bookmarksnumbered,bookmarksopen]{hyperref}
\hypersetup{%
   pdftitle   ={First exit of Brownian motion from a one-sided moving boundary},
   pdfsubject ={Manuscript},
   pdfauthor  ={Frank Aurzada and Tanja Kramm},
   pdfkeywords={}
      }

\let\BFseries\bfseries\def\bfseries{\BFseries\mathversion{bold}} 
\DeclareMathSymbol{\leqslant}{\mathalpha}{AMSa}{"36} 
\DeclareMathSymbol{\geslant}{\mathalpha}{AMSa}{"3E} 
\DeclareMathSymbol{\eset}{\mathalpha}{AMSb}{"3F}     
\newcommand{\dd}{\,\text{\rm d}}             

\newcommand{\ind}{1\hspace{-0.098cm}\mathrm{l}}

\newcommand{\IR}{\mathbb{R}}

\newcommand{\be}{\begin{eqnarray*}}
\newcommand{\ee}{\end{eqnarray*}}
\newcommand{\ben}{\begin{eqnarray}}
\newcommand{\een}{\end{eqnarray}}

\theoremstyle{plain}
\newtheorem{thm}{Theorem}
\newtheorem{lem}[thm]{Lemma}

\theoremstyle{definition}

\newtheorem{remark}[thm]{Remark}

\renewenvironment{proof}[1][] {\smallskip \noindent {\bf Proof#1.} }{\hspace*{\fill}$\square$\medskip\par}
\def\P{{\bf {\mathbb{P}}}}
\newcommand{\pr}[1]{\P\left[#1\right]}

\def\E{\mathbb{E}}

\newcommand{\indi}[1]{\,\ind_{\{#1\}}}
\def\R{\IR}
\def\d{{\rm d}}

\begin{document}
\title{First exit of Brownian motion from a one-sided moving boundary}
\author{Frank Aurzada and Tanja Kramm}
\date{\today}
\maketitle
\begin{abstract}
 We revisit a result of Uchiyama (1980): given that a certain integral test is satisfied, the rate of the probability that Brownian motion remains below the moving boundary $f$ is asymptotically the same as for the constant boundary. The integral test for $f$ is also necessary in some sense.

 After Uchiyama's result, a number of different proofs appeared simplifying the original arguments, which strongly rely on some known identities for Brownian motion. In particular, Novikov (1996) gives an elementary proof in the case of an increasing boundary. Here, we provide an elementary, half-page proof for the case of a decreasing boundary. Further, we identify that the integral test is related to a repulsion effect of the three-dimensional Bessel process. Our proof gives some hope to be generalized to other processes such as FBM.
\end{abstract}

%

\section{Introduction}
This note is concerned with the first exit time distribution of Brownian motion from a so-called moving boundary:
$$
\pr{ B_t \leq f(t), t\leq T},\qquad \text{as $T\to\infty$,}
$$
where $B$ is a Brownian motion and $f : [0,\infty)\to \R$ is the ``moving boundary''. The question we treat here is follows: for which functions $f$ does the above probability have the same asymptotic rate as in the case $f\equiv 1$? This problem was considered by a number of authors \cite{bramson1978,Uch,Gae,novikov1980,Novneu} and, besides being a classical problem for Brownian motion, has some implications for the so-called KPP equation (see e.g.\ \cite{Gae}), for branching Brownian motion (see e.g.\ \cite{bramson1978}), and for other questions.

The solution of the problem was given by Uchiyama \cite{Uch}, G\"artner \cite{Gae}, and Novikov \cite{novikov1980} independently and can be re-phrased as follows.

\begin{thm} \label{thm:main}
 Let $f : [0,\infty)\to \R$ be a continuously differentiable function with $f(0)>0$ and
\begin{equation} \label{eqn:integraltest}
 \int_1^\infty |f(t)| \,t^ {-3/2} \, \d t < \infty.
\end{equation}
Then
\begin{equation} \label{eqn:assertmain}
\pr{ B_t \leq f(t), t\leq T} \approx T^{-1/2},\qquad \text{as $T\to\infty$.}
\end{equation}
If $f$ is either convex or concave and the integral test (\ref{eqn:integraltest}) fails, $T^{-1/2}$ is not the right order in (\ref{eqn:assertmain}).
\end{thm}

Here and in the following, we denote $a(t)\approx b(t)$ if $c_1 a(t)\leq b(t) \leq c_2 a(t)$ for some constants $c_1, c_2$ and all $t$ sufficiently large.

Even though the above-mentioned problem has been solved by Uchi\-yama, there have been various attempts to simplify the proof of this result and to give an interpretation for the integral test (\ref{eqn:integraltest}). It is the purpose of this note (a) to give a simplified proof of the theorem for the case of a decreasing boundary. Our proof (b) also allows to interpret the integral test as coming from a repulsion effect of the three-dimensional Bessel process and (c) gives hope to be generalized to other processes, contrary to the existing proofs, which all make use of very specific known identities for Brownian motion.

Let us assume for a moment that $f$ is monotone. Note that the sufficiency part of the theorem can be decomposed into two parts: if $f'\geq 0$ one needs an upper bound of the probability in question, while if $f'\leq 0$ one needs a lower bound. The first case is much better studied; in particular, Novikov \cite{Novneu} gives a relatively simple proof of the theorem in this case. Contrary, in case of a decreasing boundary he wonders that ``it would be interesting to find an elementary proof of this bound'' (\cite{Novneu}, p.\ 723). We shall provide such an elementary proof here.

The remainder of this note is structured as follows.  Section~\ref{sec:proof} contains the proof of the theorem, which now fits on half a page. We also outline the relation to the Bessel process. In Section~\ref{sec:wlog}, we list some additional remarks.

\section{Proof} \label{sec:proof}
We give a proof of the following theorem, which concerns the part of Theorem~\ref{thm:main} related to the decreasing boundary.

\begin{thm} \label{thm:provedhere}
 Let $f : [0,\infty) \to \R$ be a twice continuously differentiable function with $f(0)>0$.
\begin{itemize}
 \item Then for some absolute constants $0<c_1,c_2<\infty$ we have
\begin{align*}
& \pr{ B_t \leq f(t), 0\leq t\leq T}  \\
& \geq  \pr{ B_t \leq f(0), 0\leq t\leq T} \\
& ~~~~\cdot \exp\left(-\frac{1}{2}\int_0^T f'(s)^2 \d s - c_1 \int_0^T |f''(s)| \sqrt{s}\ \d s - c_2 \sqrt{T} |f'(T)|\right). 
\end{align*}
\item In particular, if (\ref{eqn:integraltest}) holds and $f'\leq 0$, $f''\geq 0$, for large enough arguments, then we have
\begin{align}\label{eqn: asymp}
  \pr{ B_t \leq f(t), t\leq T}\approx T^{-1/2},\qquad \text{as $T\to\infty$.}
\end{align}
\end{itemize}
\end{thm}

\begin{proof}
The Cameron-Martin-Girsanov theorem implies that
\begin{align*}
& \pr{ B_t \leq f(t), 0\leq t\leq T} \notag \\
=~ &  \pr{ B_t - \int_0^t f'(s)\d s \leq f(0),0\leq  t\leq T} \notag  \\
=~ &\frac{\E [ e^{-\int_0^T f'(s) \d B_s} \indi{B_t\leq f(0), 0\leq t\leq T}]}{\pr{B_t\leq f(0),0\leq  t\leq T}}\, \cdot  \, \pr{B_t\leq f(0), 0\leq t\leq T} e^{-\frac{1}{2} \int_0^T f'(s)^2 \d s} \notag  \\
=~ & \E \left[ e^{-\int_0^T f'(s) \d B_s} \left| \sup_{0\leq t\leq T} B_t\leq f(0)\right. \right] \cdot  \pr{B_t\leq f(0),0\leq  t\leq T} e^{-\frac{1}{2} \int_0^T f'(s)^2 \d s}. 
\end{align*}
By Jensen's inequality, the first term
can be estimated from below by
\begin{equation} \label{eqn:afterjensen}
 \exp\left( \E \left[ - \int_0^T f'(s) \d B_s \left| \sup_{0\leq t\leq T} B_t\leq f(0)\right. \right] \right) = \exp\left( - \E \int_0^T f'(s) \d Y_s \right),
\end{equation}
where we denote by $Y$ the law of $B$ conditioned on $\sup_{0\leq t\leq T} B_t\leq f(0)$. Further,
\begin{eqnarray*}
\int_0^T f'(s) \d Y_s & =&
\int_0^T \big[ \int_0^s f''(u) \d u + f'(0) \big] \d Y_s\\
& =&\int_0^T (Y_T-Y_u) f''(u)\d u  + f'(0) Y_T\\
& =&-\int_0^T Y_u f''(u) \d u  +Y_T f'(T).
\end{eqnarray*}
so that the term in (\ref{eqn:afterjensen}) equals
\begin{equation} \label{eqn:noreg}
\exp\left( \int_0^T \E Y_u \, f''(u) \d u  +  \E Y_T \, (-f'(T))\right),
\end{equation}
and hence the theorem is proved as soon as Lemma~\ref{lem:bessellem} below is shown.
\end{proof}

\begin{lem} \label{lem:bessellem}
 Let $B$ be a Brownian motion and $f(0)>0$ be some constant. Then there is a constant $c$ such that
$$
\E \left[ B_u \left| \sup_{0\leq t\leq T} B_t\leq f(0)\right. \right] \geq - c \sqrt{u},\qquad \forall 0\leq u\leq T.
$$
\end{lem}

In order to show this lemma one can use the distribution of maximum over an interval and terminal value of Brownian motion, which is explicitly known (see e.g.\ \cite{ks}, Prop.\ 2.8.1). However, we do not include this proof here. Let us rather mention that the lemma can also be seen through a relation to the three-dimensional Bessel process, as detailed now.

Recall that a (three-dimensional) Bessel process has three representations: it can be defined firstly as Brownian motion conditioned to be positive for all times, secondly as the solution of a certain stochastic differential equation (which gives rise to Bessel processes of other dimensions), and thirdly as the modulus of a three-dimensional Brownian motion, see e.g.\ \cite{ks}, Chapter 3.3.C. If we denote by $Y$ the law of a Brownian motion $B$ under the conditioning $\{ \sup_{0\leq t\leq T} B_t\leq f(0) \}$, it seem intuitively clear that one can find a Bessel process $-X$ such that $Y\leq X$, using the first representation of $-X$. Now, taking expectations and using the third representation of $-X$ (and Brownian motion scaling) it is clear that $\E Y_s \geq \E X_s = - c \sqrt{s}$. Thus, the integral test is related to the repulsion of Brownian motion by the conditioning.

\section{Further remarks} \label{sec:wlog}

\begin{remark}
 Clearly the value of $f$ in a finite time horizon $[0,t_0]$ does not matter for the outcome of the problem, as we are interested in asymptotic results. Any finite time horizon can be cut off with the help of Slepian's inequality \cite{Sle}:
$$
\pr{ B_t \leq f(t), 0\leq t\leq T}  \geq \pr{ B_t \leq f(t), 0\leq t\leq t_0} \cdot \pr{ B_t \leq f(t), t_0\leq t\leq T}.
$$
\end{remark}

\begin{remark}
Let us comment on the regularity assumptions: it is clear that these are of technical matter and of no importance to the question. Note that one can easily modify a regular function $f$ such that either (\ref{eqn:integraltest}) fails or (\ref{eqn:assertmain}) does not hold. The only way to avoid pathologies and to prove a general result is to assume regularity. Note that the theorem is obviously true if we replace $f$ by a whatever irregular function $g$ with $f\leq g$. The same can be said about the monotonicity/convexity assumption in the second part of Theorem~\ref{thm:provedhere}.
\end{remark}

\begin{remark}
 It is easy to see that the integral test (\ref{eqn:integraltest}) implies
$$
 \int_0^\infty  f'' (s) s^{1/2} \dd s < \infty\qquad\text{and}\qquad \int_0^\infty  f' (s)^2 \dd s < \infty
$$
under the assumption of monotonicity and concavity.
\end{remark}

\begin{remark} \label{rem:novikov}
 Thanks to \cite{Novneu}, Theorem 2, if (\ref{eqn:integraltest}) holds one does not only obtain (\ref{eqn: asymp}) but also the strong asymptotic order
\begin{align*}
  \lim_{T\to\infty} T^{1/2} \pr{ B_t \leq f(t), t\leq T} = \sqrt{\frac2\pi}~ \E B_\tau,
\end{align*}
where $0< \E B_\tau=\E f(\tau)<\infty$ with $\tau := \inf \{t >0: B_t = f(t) \}$.
\end{remark}

\begin{remark}
 The last remark concerns possible generalizations to other processes. Note that the technique of the main proof (Girsanov's theorem, Jensen's inequality) does carry over to other processes. The crucial point is determining the repulsion effect of the conditioning in Lemma~\ref{lem:bessellem}. The authors do not see at the moment how a similar lemma can be established for processes other than Brownian motion, e.g.\ FBM.
\end{remark}

\noindent {\bf Acknowledgement:} Frank Aurzada and Tanja Kramm were supported by the DFG Emmy Noether programme.

\bibliographystyle{abbrv}

\begin{thebibliography}{1}

\bibitem{bramson1978}
M.~D. Bramson.
\newblock Maximal displacement of branching {B}rownian motion.
\newblock {\em Comm. Pure Appl. Math.}, 31(5):531--581, 1978.

\bibitem{Gae}
J.~G{\"a}rtner.
\newblock Location of wave fronts for the multidimensional {KPP} equation and
  {B}rownian first exit densities.
\newblock {\em Math. Nachr.}, 105:317--351, 1982.

\bibitem{ks}
I.~Karatzas and S.~E. Shreve.
\newblock {\em Brownian motion and stochastic calculus}, volume 113 of {\em
  Graduate Texts in Mathematics}.
\newblock Springer-Verlag, New York, second edition, 1991.

\bibitem{novikov1980}
A.~A. Novikov.
\newblock A martingale approach to first passage problems and a new condition
  for {W}ald's identity.
\newblock In {\em Stochastic differential systems ({V}isegr\'ad, 1980)},
  volume~36 of {\em Lecture Notes in Control and Information Sci.}, pages
  146--156. Springer, Berlin, 1981.

\bibitem{Novneu}
A.~A. Novikov.
\newblock Martingales, a {T}auberian theorem, and strategies for games of
  chance.
\newblock {\em Teor. Veroyatnost. i Primenen.}, 41(4):810--826, 1996.

\bibitem{Sle}
D.~Slepian.
\newblock The one-sided barrier problem for {G}aussian noise.
\newblock {\em Bell System Tech. J.}, 41:463--501, 1962.

\bibitem{Uch}
K.~Uchiyama.
\newblock Brownian first exit from and sojourn over one-sided moving boundary
  and application.
\newblock {\em Z. Wahrsch. Verw. Gebiete}, 54(1):75--116, 1980.

\end{thebibliography}

\end{document}